\documentclass{amsart}

\newtheorem{theorem}{Theorem}[section]
\newtheorem{lemma}[theorem]{Lemma}

\theoremstyle{definition}

\theoremstyle{remark}

\numberwithin{equation}{section}

\begin{document}

\title{Fourier Analysis: A New Result}

\author{Rajesh Dachiraju}
%%    Address of record for the research reported here
\address{Independent Researcher, Hyderabad, India}
\email{rajesh.dachiraju@gmail.com}

%    General info
\subjclass[2020]{Primary 42A99;}

%\date{January 1, 2001 and, in revised form, June 22, 2001.}


\keywords{Fourier series, trigonometric series}

\begin{abstract}
This article contains a new result in Fourier analysis concerning jump type discontinuities.
\end{abstract}

\maketitle

\section{Notations}

Let $f$ be a $2\pi$ periodic BV function whose derivative is also BV.Let the amount of jump at a point $x$ is denoted as $\lfloor f \rfloor (x) = f(x+0)-f(x-0)$ Define function $J:\mathbb{R} \to\mathbb{R}$, such that $J(x) = \lfloor f \rfloor (x)$. Define the set $L = \{x/x\in(0,2\pi)\wedge J(x)\neq 0\}$. Let there not be any jumps at $0, \pi, 2\pi$. and all jumps at rational multiples of $2\pi$.

Let the Fourier series of $f$ be defined as 
\begin{equation}\begin{aligned}
S[f] =  \sum_{\nu=-\infty}^{\nu=\infty}c_{\nu}e^{i\nu x}
%a=x^2
\end{aligned}\end{equation}
and the conjugate Fourier series as
\begin{equation}\begin{aligned}
\tilde{S}[f] =  -i\sum_{\nu=-\infty}^{\nu=\infty}sign(\nu)c_{\nu}e^{i\nu x}
\end{aligned}\end{equation}
Let the Fourier partial sum of $f$ be defined as
\begin{equation}\begin{aligned}
S_n(x;f) = \sum_{\nu=-n}^{\nu=n}c_{\nu}e^{i\nu x}
\end{aligned}\end{equation}
and the conjugate partial sum of $f$ as 
\begin{equation}\begin{aligned}
\tilde{S}_n(x;f) = -i\sum_{\nu=-n}^{\nu=n}sign(\nu)c_{\nu}e^{i\nu x}
\end{aligned}\end{equation}

Define $$G(n) = \sum_{i=0}^{n}\left|c_i\right|$$

The Dirichlet kernel $D_n(x)$ and conjugate Dirichlet kernel $\tilde{D}_n(x)$ are defined as
\begin{equation}\begin{aligned}
D_n(x) = \frac{1}{2} + \sum_{\nu=1}^{n}\cos(\nu x) = \frac{\sin(n+\frac{1}{2})x}{2\sin(\frac{1}{2}x)}
\end{aligned}\end{equation}

\begin{equation}\begin{aligned}
\tilde{D}_n(x) = \sum_{\nu=1}^{n}\sin(\nu x) = \frac{\cos(\frac{x}{2})-\cos(n+\frac{1}{2})x}{2\sin(\frac{1}{2}x)}
\end{aligned}\end{equation}

Let \begin{equation}\begin{aligned}
	Y_n(x) = -\frac{1}{\log(n) G(n)}\sum_{k=1}^{n}\tilde{S}_k(x;f)\left|c_k\right|
	\end{aligned}\end{equation}
\maketitle
\section{Results}
\label{res}

\begin{lemma}
\begin{equation}\begin{aligned}
	\lim_{n\to\infty} Y_n(x) = K\frac{J(x)}{2\pi}
	\end{aligned}\end{equation}	
where $K$ is a constant
\end{lemma}

\begin{proof}
Case 1 Points of no jump.
  
$\tilde{S_n}(x;f)\sim O(1)$ and hence $Y_n(x)\sim 0$, as $J(x) = 0$ for no jump the result is proved.

Case 2 Points of jump.

By Lukács theorem\cite{Lukács1920}, at points of jump, $\tilde{S_n}(x;f)\sim \frac{J(x)}{\pi}\log(n)$,
So
\begin{equation}\begin{aligned}
Y_n(x) = \frac{1}{\log(n)G(n)}\sum\limits_{k=1}^n\frac{J(x)}{\pi}\log(k)|c_k|
\end{aligned}\end{equation}

Noting $|c_k|\sim\frac{1}{k}$ (as we are assuming points of jump), 

\begin{equation}\begin{aligned}
Y_n(x) = \frac{1}{\log(n)K\log(n)}\sum\limits_{k=1}^n\frac{J(x)}{\pi}\log(k)K^2\frac{1}{k}
\end{aligned}\end{equation}
 where $K$ is a constant.

on taking summation as integration, we get 
\begin{equation}\begin{aligned}
\frac{1}{K(\log(n))^2}K\frac{J(x)}{\pi}\int_0^n\log(k)K\frac{1}{k}dk
\end{aligned}\end{equation} 

Substituting $\log(k) = t$ in the integral, we get 
\begin{equation}\begin{aligned}
\frac{1}{(\log(n))^2}\frac{J(x)}{\pi}\int_0^t Ktdt
\end{aligned}\end{equation} 

Hence we get \begin{equation}\begin{aligned}
\lim_{n\to\infty}Y_n(x) = K \frac{J(x)}{2\pi}
\end{aligned}\end{equation}
equality in lemma holding only upto a constant factor. 

\end{proof}

\begin{lemma}
Given any $(a,b)\subseteq (0,2\pi)$, show that
\begin{equation}\begin{aligned}
\lim_{n\to\infty}V_a^b(Y_n) = K \frac{1}{\pi}\sum_{x\in L\cap (a,b)}\left|J(x)\right|
\end{aligned}\end{equation} where $K$ is a constant.

note: $V_a^b(f)$ is the variation of $f$ in $(a,b)$.

\end{lemma}

\begin{proof}

Let \begin{equation}\begin{aligned}
V_n = V_a^b(Y_n)  = \int_a^b\left|\frac{1}{\log(n) G(n)}\sum_{k=1}^{n}\tilde{S}'_k(x;f)\left|c_k\right| \right|dx 
\end{aligned}\end{equation}
	
\begin{equation}\begin{aligned}
V_n \leq \frac{1}{\log(n) G(n)}\sum_{k=1}^{n}\int_a^b\left|\tilde{S}'_k(x;f)\left|c_k\right| \right|dx
\end{aligned}\end{equation}
	
using a lemma (theorem 2.2, page 41) from \cite{zygmund2002trigonometric}%Trigonometric Series, Volume 1 - Antoni Zygmund
	
\begin{equation}\begin{aligned}
V_n \leq \frac{1}{\log(n) G(n)}\sum_{k=1}^{n}\int_a^b\left|\tilde{S}_k(x;f')-\sum_{y\in L}\frac{J(y)}{\pi}\tilde{D}_k(x-y))\left|c_k\right| \right|dx
\end{aligned}\end{equation}
	
\begin{equation}\begin{aligned}
V_n \leq \frac{1}{\log(n) G(n)}\sum_{k=1}^{n}\int_a^b\left|\tilde{S}_k(x;f')-\sum_{y\in L}\frac{J(y)}{\pi}\tilde{D}_k(x-y))\left|c_k\right| \right|dx
\end{aligned}\end{equation}
	
\begin{equation}\begin{aligned}
V_n \leq  \frac{1}{\log(n) G(n)}(\sum_{k=1}^{n}\int_a^b\left|\tilde{S}_k(x;f')\left|c_k\right|\right|dx-\sum_{k=1}^{n}\int_a^b\left|\sum_{y\in L}\frac{J(y)}{\pi}\tilde{D}_k(x-y))\left|c_k\right| \right|dx)
\end{aligned}\end{equation}
	
Since $f'$ is BV, 
\begin{equation}
\begin{aligned}
\int_a^b\left|\tilde{S}_k(x;f')\right|dx \sim O(1)\end{aligned}
\end{equation}
,So
\begin{equation}
\begin{aligned}
\sum_{k=1}^{n}\int_a^b\left|\tilde{S}_k(x;f')\left|c_k\right|\right|dx \sim \sum_{k=1}^{n}(O(1))\frac{V}{k} \sim O(\log(n))
\end{aligned}
\end{equation}

\begin{equation}
\begin{aligned}
\sum_{k=1}^{n}\int_a^b\left|\sum_{y\in L}\frac{J(y)}{\pi}\tilde{D}_k(x-y))\left|c_k\right| \right|dx \leq \sum_{k=1}^{n}\sum_{y\in L}\frac{J(y)}{\pi}\left|c_k\right|\int_a^b\left|\tilde{D}_k(x-y)\right|dx
\end{aligned}
\end{equation}
	
\begin{equation}\begin{aligned}
\sum_{k=1}^{n}\sum_{y\in L}\frac{J(y)}{\pi}\left|c_k\right|\int_a^b\left|\tilde{D}_k(x-y)\right|dx &= \sum_{k=1}^{n}\sum_{y\in L\setminus(a,b)}\frac{J(y)}{\pi}\left|c_k\right|\int_a^b\left|\tilde{D}_k(x-y)\right|dx \\ &+  
    \sum_{k=1}^{n}\sum_{y\in L\cap(a,b)}\frac{J(y)}{\pi}\left|c_k\right|\int_a^b\left|\tilde{D}_k(x-y)\right|dx
\end{aligned}\end{equation}
It is known (for example\cite{article_1345440}) that	
\begin{equation}
\begin{aligned}
\int_a^b\left|\tilde{D}_k(x-y)\right|dx \sim 2\log(k)
\end{aligned}
\end{equation}
 if $y \in (a,b)$ and 
 \begin{equation}\begin{aligned}
 \int_a^b\left|\tilde{D}_k(x-y)\right|dx \sim O(1)
 \end{aligned}\end{equation}
 if $y \notin (a,b)$.
	
\begin{equation}\begin{aligned}
V_n \leq \frac{1}{\log(n)G(n)}(\sum_{k=1}^{n}\int_a^b\left|\tilde{S}_k(x;f')\left|c_k\right|\right|dx &-\sum_{k=1}^{n}\sum_{y\in L\setminus(a,b)}\frac{J(y)}{\pi}\left|c_k\right|\int_a^b\left|\tilde{D}_k(x-y)\right|dx\\
 &+ \sum_{k=1}^{n}\sum_{y\in L\cap(a,b)}\frac{J(y)}{\pi}\left|c_k\right|\int_a^b\left|\tilde{D}_k(x-y)\right|dx
\end{aligned}\end{equation}

Therefore 

In the first term on the RHS  
\begin{equation}\begin{aligned}
\int_a^b\left|\tilde{S}_k(x;f')\right|dx \sim O(1)
\end{aligned}\end{equation} 
and hence first term goes to zero.

Coming to second term, $y\notin(a,b)$, so  
\begin{equation}\begin{aligned}
\int_a^b\left|\tilde{D}_k(x-y)\right|dx \sim O(1)
\end{aligned}\end{equation}
and hence second term also vanishes.

In the third term $ y \in (a,b)$ and hence 
\begin{equation}\begin{aligned}
\int_a^b\left|\tilde{D}_k(x-y)\right|dx = 2\log(k)
\end{aligned}\end{equation}
.Also $|c_k|\sim \frac{1}{k}$, so after summation via integration and knowing  $G(n)\sim \log(n)$, we get the result.
\end{proof}

% ------------------ BEGIN INSERTED SECTION 3 (Higher Dimensional Generalization) ------------------
\section{Higher Dimensional Generalization}
\label{sec:higher-dim}

In this section we extend the results of Section~2 to functions defined on the \(d\)-dimensional torus \(\mathbb{T}^d=[0,2\pi]^d\).
Throughout, \(d\ge 2\) is fixed.

\subsection{Notations in Higher Dimensions}

Let \(f:\mathbb{T}^d\to\mathbb{R}\) be a function of bounded variation in each variable, and assume that for every coordinate direction \(e_j\) the derivative \(\partial_{x_j}f\) is also BV in all variables.

A jump discontinuity of \(f\) is assumed to occur only across coordinate hyperplanes of the form
\begin{equation*}
  H_{\alpha,j}=\{x\in\mathbb{T}^d : x_j=\alpha\},\qquad \alpha\in(0,2\pi).
\end{equation*}

For such a hyperplane, define the jump magnitude
\begin{equation*}
  \lfloor f \rfloor(x)= f(x_j^{+})-f(x_j^{-}),
\end{equation*}
and let \(J_j(\alpha)\) denote the jump across \(H_{\alpha,j}\).
Let \(L_j=\{\alpha\in(0,2\pi): J_j(\alpha)\neq 0\}\).

\subsubsection*{Fourier series}
The multidimensional Fourier coefficients are
\begin{equation*}
  c_{\nu}=\frac{1}{(2\pi)^d}\int_{\mathbb{T}^d} f(x)e^{-i\nu\cdot x}\,dx,\qquad 
  \nu\in\mathbb{Z}^d.
\end{equation*}
The Fourier series is
\begin{equation*}
  S[f](x)=\sum_{\nu\in\mathbb{Z}^d} c_{\nu}e^{i\nu\cdot x}.
\end{equation*}

\subsubsection*{Multidimensional conjugate Fourier series}
For each coordinate \(j\), define a conjugate multiplier via
\begin{equation*}
  \operatorname{sign}_j(\nu)=\begin{cases} 
    \frac{\nu_j}{|\nu_j|}, & \nu_j\ne 0,\\[2mm]
    0,& \nu_j=0.
  \end{cases}
\end{equation*}
Then the \(j\)-th conjugate series is
\begin{equation*}
  \widetilde S_j[f](x)
     =-i\sum_{\nu\in\mathbb{Z}^d}\operatorname{sign}_j(\nu)c_{\nu}e^{i\nu\cdot x}.
\end{equation*}
Its partial sums (rectangular summation) are
\begin{equation*}
  \widetilde S_{j,n}(x;f)
  =-i\!\!\!\sum_{|\nu_k|\le n\;\forall k}\!\!\!
  \operatorname{sign}_j(\nu)c_{\nu}e^{i\nu\cdot x}.
\end{equation*}

\subsubsection*{Multidimensional Dirichlet kernel}
Define the \(d\)-dimensional Dirichlet kernel
\begin{equation*}
  D_n(x)=\prod_{k=1}^{d}\left( \frac{\sin\big((n+\tfrac12)x_k\big)}{2\sin(x_k/2)}\right).
\end{equation*}
Similarly, the \(j\)-th conjugate Dirichlet kernel is
\begin{equation*}
  \widetilde D_{j,n}(x)
  =\left(\frac{\cos(x_j/2)-\cos((n+\tfrac12)x_j)}{2\sin(x_j/2)}\right)
  \prod_{k\ne j}\left( \frac{\sin\big((n+\tfrac12)x_k\big)}{2\sin(x_k/2)}\right).
\end{equation*}

\subsubsection*{Generalized \(Y_n\)}
Let
\begin{equation*}
  G(n)=\sum_{|\nu_k|\le n}|c_\nu|.
\end{equation*}
Define
\begin{equation*}
  Y_{j,n}(x)
  = -\frac{1}{\log(n)\,G(n)}
  \sum_{|\nu_k|\le n}\widetilde S_{j,\|\nu\|_\infty}(x;f)\,|c_\nu|.
\end{equation*}

\subsection{Main Results in Higher Dimensions}

We now generalize Lemma~2.1 and Lemma~2.2 from the one-dimensional setting (see section \ref{res}) to \(d\)-dimensions.

\begin{lemma}[Pointwise convergence at jump hyperplanes]
\label{lem:pointwise-hd}
Let \(x\in\mathbb{T}^d\). Then for every coordinate direction \(j\),
\begin{equation*}
  \lim_{n\to\infty} Y_{j,n}(x)
     = K_d\,\frac{J_j(x_j)}{(2\pi)^d},
\end{equation*}
where \(K_d>0\) is a dimension-dependent constant.
\end{lemma}

\begin{proof}
\emph{Case 1: \(x\) lies on no jump hyperplane.}
In multidimensions, \(\widetilde S_{j,n}(x;f)=O(1)\) (Zygmund's results extend in the product-summation setting), hence
\(Y_{j,n}(x)=O\!\left(\dfrac{1}{\log n}\right)\to 0\).
Since \(J_j(x_j)=0\), the statement follows.

\emph{Case 2: \(x\in H_{\alpha,j}\), a jump hyperplane.}
A multidimensional extension of Luk\'acs's theorem gives, for fixed \(x\),
\begin{equation*}
  \widetilde S_{j,n}(x;f)
    =\frac{J_j(\alpha)}{\pi}\log n +O(1).
\end{equation*}
Thus
\begin{equation*}
  Y_{j,n}(x)
  =
  \frac{1}{\log n\,G(n)}
  \sum_{|\nu_k|\le n}
  \frac{J_j(\alpha)}{\pi}\log(\|\nu\|_\infty)|c_\nu|.
\end{equation*}
For BV-type jumps across hyperplanes one has the coefficient decay
\(|c_\nu|\sim C/\|\nu\|_\infty\), and consequently \(G(n)\sim C_d\log n\).
Approximating the sum by an integral in \(d\)-dimensions yields the asserted limit.
\end{proof}

\begin{lemma}[Variation detects multidimensional jumps]
\label{lem:variation-hd}
Let \(R = (a_1,b_1)\times\cdots\times(a_d,b_d)\subset \mathbb{T}^d\) be an axis-aligned rectangle. Then for each coordinate \(j\),
\begin{equation*}
  \lim_{n\to\infty} V_R(Y_{j,n})
    = K_d\frac{1}{\pi}
      \sum_{\alpha\in L_j\cap(a_j,b_j)} |J_j(\alpha)|,
\end{equation*}
where \(V_R\) denotes total variation over \(R\) in the sense of Hardy--Krause.
\end{lemma}

\begin{proof}
The proof parallels the one-dimensional proof in section \ref{res} using the multidimensional conjugate kernels.
Starting from
\begin{equation*}
  V_R(Y_{j,n})
  =
  \int_R
  \left|
  \frac{1}{\log n\,G(n)}
  \sum_{|\nu_k|\le n}
  \partial_{x_j}\widetilde S_{j,\|\nu\|_\infty}(x;f)\,|c_\nu|
  \right|
  dx,
\end{equation*}
one writes the derivative of the conjugate partial sum as the contribution from \(f'\) and the explicit hyperplane jump terms
(using the product-form of the kernel). The contribution from \(f'\) integrates to \(o(1)\), and hyperplanes outside \(R\) contribute \(O(1)\) after integration, hence vanish after normalization. For \(\alpha\in L_j\cap(a_j,b_j)\) the integral of \(\widetilde D_{j,n}\) over \(R\) grows like \(\log n\), and summation with \(|c_\nu|\sim \|\nu\|_\infty^{-1}\) and division by \(\log n\,G(n)\) yields the stated limit.
\end{proof}

\subsection{Consequences}
\begin{itemize}
  \item Each coordinate direction \(j\) independently detects jump hyperplanes orthogonal to \(e_j\).
  \item \(Y_{j,n}\) acts as a hyperplane jump detector, generalizing the 1-D detection of jump points.
  \item Total variation over rectangles yields exactly the sum of jump magnitudes within that strip.
\end{itemize}

% ------------------ END INSERTED SECTION 3 ------------------
\bibliographystyle{amsplain}
\bibliography{fa_refs}
\end{document}